\documentclass[a4 paper, 11pt]{amsart}
\usepackage[utf8]{inputenc}
\usepackage{amssymb, hyperref, enumerate, color}

\usepackage{amssymb}
\usepackage{amsthm}
\usepackage{amsmath}
\usepackage{enumerate}
\usepackage{mathrsfs}
\usepackage{graphicx}
\usepackage{bookmark}
\usepackage{tikz}
\usetikzlibrary{decorations.markings}

\usepackage{floatrow}

\usepackage{hyperref}

\theoremstyle{theorem}
\newtheorem{theorem}{Theorem}[section]
\newtheorem{proposition}[theorem]{Proposition}
\newtheorem{lemma}[theorem]{Lemma}
\newtheorem*{lemma*}{Lemma}

\theoremstyle{definition}

\newtheorem{opq}[theorem]{Open Question}

\theoremstyle{remark}
\newtheorem{remark}[theorem]{Remark}

\numberwithin{equation}{section}

\newcommand{\R}{\mathbb{R}}
\newcommand{\N}{\mathbb{N}}
\newcommand{\C}{\mathbb{C}}

\newcommand{\ut}{\tilde{u}}
\newcommand{\vt}{\tilde{v}}
\newcommand{\ft}{\tilde{f}}
\newcommand{\gt}{\tilde{g}}
\newcommand{\at}{\tilde{a}}

\newcommand{\xt}{(u,v,w,\ut,\vt)}

\newcommand{\esl}{e^{\sqrt{\lambda}}}
\newcommand{\nesl}{e^{-\sqrt{\lambda}}}
\newcommand{\lb}{\lambda}

\newcommand{\slb}{\sqrt{\lambda}}

\DeclareMathOperator{\real}{Re}

\DeclareMathOperator{\Ker}{Ker}
\DeclareMathOperator{\Ran}{Ran}

\title[Energy decay in a 1-D wave-heat-wave system]{Optimal energy decay in a one-dimensional wave-heat-wave system}

\author[A.C.S.\ Ng]{Abraham C.S.\ Ng}
\address[A.C.S.\ Ng]{St Edmund Hall, Queen's Lane, Oxford OX1 4AR, UK}
\email{abraham.ng@maths.ox.ac.uk}

\begin{document}
	
	\begin{abstract}
		Harnessing the abstract power of the celebrated result due to Borichev and Tomilov (Math.\ Ann.\ 347:455--478, 2010, no.\ 2), we study the energy decay in a one-dimensional coupled wave-heat-wave system. We obtain a sharp estimate for the rate of energy decay of classical solutions by first proving a growth bound for the resolvent of the semigroup generator and then applying the asymptotic theory of $C_0$-semigroups. The present article can be naturally thought of as an extension of a recent paper by Batty, Paunonen, and Seifert (J.\ Evol.\ Equ.\ 16:649--664, 2016) which studied a similar wave-heat system via the same theoretical framework.
	\end{abstract}
	
	
	\maketitle
	
\section{Introduction}

In this article, we apply the theorem of Borichev-Tomilov \cite[Theorem 4.1]{BoTo} to a one-dimensional system with coupled wave and heat parts. This application is modelled upon the 2016 paper of Batty, Paunonen, and Seifert \cite{BPS1} where the `optimal energy decay in a one-dimensional coupled wave-heat system' with finite Neumann wave and Dirichlet heat parts was studied by analysing the following system:

\begin{equation}\label{WHeq1}
\begin{cases}\begin{aligned} & u_{tt}(\xi,t) = u_{\xi\xi}(\xi,t), & \xi \in (-1,0), \ & t>0,\\
& w_t(\xi,t) = w_{\xi\xi}(\xi,t), & \xi \in (0,1), \ & t>0, \\
& u_t(0,t) = w(0,t), \ \ \ u_\xi(0,t) = w_\xi(0,t), & & t>0, \\
& u_\xi(-1,t) = 0, \ \ \ w(1,t) = 0, & & t>0, \\
& u(\xi,0)=u(\xi), \ \ \ u_t(\xi,0) = v(\xi) \ & \xi \in (-1,0), \\
& w(\xi,0) = w(\xi), \ & \xi \in (0,1),
\end{aligned}\end{cases}
\end{equation}
where the initial data $u,v,$ and $w$ lived in $H^1(-1,0), L^2(-1,0)$ and $L^2(0,1)$ respectively. The energy was then defined, given a vector of initial data $x=(u,v,w)$, as
$$E_x(t) = \frac{1}{2}\int_{-1}^1 |u_\xi(\xi,t)|^2 + |u_t(\xi,t)|^2 + |w(\xi,t)|^2 \ d\xi, \ \ \ t\geq 0,$$ with all the functions being understood to have been extended by zero in $\xi$ to the interval $(-1,1)$. If the solution is sufficiently regular, a routine calculation via integration by parts shows that
$$E'_x(t) = -\int_0^1|w_\xi(\xi,t)|^2\ d\xi, \ \ \ t \geq 0,$$ and, in particular, that the energy of any such solution is non-increasing with respect to time. The main goal of analysing such a model is to quantitatively estimate the rate of energy decay of a given solution.

The system (\ref{WHeq1}) was first studied (with Dirichlet boundary at $\xi=-1$ and a slightly different coupling condition) in \cite{ZZ1}, yielding the sharp decay rate $E_x(t) = O(t^{-4}), t\to\infty$ (see below for the meaning of `big O' notation). The approach in \cite{ZZ1} relied on a rather complicated spectral analysis used in conjunction with the theory of Riesz spectral operators. In contrast to \cite{ZZ1}, however, the approach in \cite{BPS1} was based on the semigroup methods of non-uniform stability pioneered by Batty and Duyckaerts in \cite{BaDu}, widely popularised by Borichev and Tomilov in \cite{BoTo}, and largely completed by Rozendaal, Seifert, and Stahn in \cite{RSS}, greatly simplifying the analysis necessary to obtain the rate of decay.

The motivation of studying models like this and, in particular, the one in this article presented below, stems mainly from the study of fluid-structure models where, often in higher-dimensional settings, the Navier-Stokes equations (the fluid half) are coupled with the nonlinear elasticity equation (the structure half). We refer to \cite[Section 1]{BPS1} and \cite{AvTr} for surveys of similar problems (see also \cite{BPS2} where the same approach with suitable adjustments is applied to study a wave-heat system on a rectangular domain).

In this article, we add an extra wave component to the system (\ref{WHeq1}) and take Dirichlet boundary conditions on both ends, analysing the following wave-heat-wave system:
%In the rest of the chapter below, we modify (\ref{WHeq1}) in two different ways. The second of these two investigations ends up being especially distinct from that of \cite{BPS1}, requiring a number of different techniques to analyse.
%\section{Optimal Energy Decay in a One-Dimensional Coupled Wave-Heat-Wave System}%WHW
\begin{equation}\label{WHWeq1}
\begin{cases}\begin{aligned} & u_{tt}(\xi,t) = u_{\xi\xi}(\xi,t), & \xi \in (0,1), \ & t>0,\\
& w_t(\xi,t) = w_{\xi\xi}(\xi,t), & \xi \in (1,2), \ & t>0, \\
& \ut_{tt}(\xi,t) = \ut_{\xi\xi}(\xi,t), & \xi \in (2,3), \ & t>0, \\
& u(0,t) = \ut(3,t) = 0, & & t>0, \\
& u_t(1,t) = w(1,t),\ \ \ u_\xi(1,t) = w_\xi(1,t), & & t>0, \\
& \ut_t(2,t) = w(2,t),\ \ \ \ut_\xi(2,t) = w_\xi(2,t), & & t>0, \\
& u(\xi,0) = u(\xi), \ \ \ u_t(\xi,0) = v(\xi), & \xi \in (0,1), \ & \\
& w(\xi,0) = w(\xi), & \xi \in (1,2), \ & \\
& \ut(\xi,0) = \ut(\xi), \ \ \ \ut_t(\xi,0) = \vt(\xi), & \xi \in (2,3). \ & \\
\end{aligned}\end{cases}
\end{equation}
The initial data is required to satisfy $u = u(\xi,0) \in H^1(0,1), v = u_t(\xi,0) \in L^2(0,1), w = w(\xi,0) \in L^2(1,2), \ut = \ut(\xi,0) \in H^1(2,3),$ and $\vt = \ut_t(\xi,0) \in L^2(2,3)$.

As in \cite{BPS1}, the aim here is to find a quantitative estimate for the rate of energy decay of a given solution. Given a vector of initial data $x=(u,v,w,\ut,\vt)$ satisfying the conditions above, we similarly define the energy of the corresponding solution as 
$$E_x(t) = \frac{1}{2}\int_0^3 |u_\xi(\xi,t)|^2 + |u_t(\xi,t)|^2 + |w(\xi,t)|^2 + |\ut_\xi(\xi,t)|^2 + |\ut_t(\xi,t)|^2 \ d\xi, \ \ \ t\geq 0.$$ Again, all functions have been extended by zero in $\xi$ to the interval $(0,3)$. Provided we have sufficient regularity of the solution, a simple calculation via integration by parts shows that
$$E'_x(t) = \real{\left\{\ut_\xi(3,t)\overline{\ut_t(3,t)} - u_\xi(0,t)\overline{u_t(0,t)}\right\}} - \int_1^2|w_\xi(\xi,t)|^2\ d\xi, \ \ \ t\geq0.$$ Since $u_t(0,t) = \frac{\partial}{\partial t}u(0,t) = \ut_t(3,t)  = \frac{\partial}{\partial t}\ut(3,t) = 0$ for $t >0$, the energy of any such solution is non-increasing with respect to time. %Hence, the energy of any such solution is non-increasing in time if the condition
%$$\mathcal{C}(u(\cdot,t),u_t(\cdot,t),\ut(\cdot,t),\ut_t(\cdot,t)) \leq 0, \ \ \ t \geq0,$$ is satisfied where
%$$\mathcal{C}(u,v,\ut,\vt):=\real{\left\{\ut_\xi(3)\overline{\vt(3,t)} - u_\xi(0,t)\overline{v(0,t)}\right\}}$$ for any $u,v,\ut,\vt$ in $H^1(0,1),L^2(0,1),H^1(2,3),L^2(2,3)$ respectively.
The remaining sections are devoted to obtaining a sharp quantitative estimate for the rate of this decay for classical solutions of (\ref{WHWeq1}), but first, we detail below, the mostly standard notation used in this article.

Closely following the notation of \cite{BPS1}, the domain, kernel, range, spectrum, and range of a closed operator $A$ acting on a Hilbert space (always complex by assumption) will be denoted by $D(A), \Ker{A}, \Ran{A}, \sigma(A)$ and $\rho(A)$ respectively. For $\lb \in\rho(A)$, we write $R(\lb,A)$ to signify the resolvent operator $(\lb-A)^{-1}$. For $\lb \in \C$, we define the square root $\sqrt{\lb}$ by taking the branch cut along the negative real axis, that is, for $\lb =re^{i\theta}$ where $r\geq0$ and $\theta \in (-\pi,\pi]$, we let $\sqrt{\lb} = r^{1/2}e^{i\theta/2}$. We also denote the closed complex left half-plane by $\C_- := \{z \in \C : \real{z} <0\}$. Finally, given two functions $f,g:(0,\infty) \to [0,\infty]$ and $a \in [0,\infty]$ fixed, we write $f(t) = O(g(t)),\ t\to \infty$, to indicate that there exists some constant $C>0$ such that $f(t)\leq Cg(t)$ for all $t$ sufficiently large, the so-called `big O notation'. If $g$ is strictly positive for all sufficiently large $t>0$, we write $f(t)=o(g(t)),\ t \to \infty$, to mean that $f(t)/g(t) \to 0$ as $t\to \infty$, the so-called `little o notation'. If $p$ and $q$ are non-negative real-valued quantities, the notation $p \lesssim q$ denotes that $p\leq Cq$ for some constant $C>0$ that is independent of any varying parameters in a given context.

\subsection*{Acknowledgements}
The author thanks David Seifert and Charles Batty for helpful discussions on the topic of this article and is especially indebted to David for his careful reading and feedback of several drafts of this article. The author is also grateful to the University of Sydney for funding this work through the Barker Graduate Scholarship.

% and we show in the main result that the energy of any classical solution to (\ref{WHWeq1}) satisfies
%$$E_x(t) = o(t^{-4}), \ \ \ t \to \infty.$$ Moreover, we prove that this rate is optimal. This shows that adding an extra wave equation to the coupled wave-heat system of \cite{1} preserves the rate of decay up to a change in constant.

\section{Well-posedness -- the Semigroup and its Generator}

In this section, we first prove that (\ref{WHWeq1}) is well posed and has solution given by the orbits of a $C_0$-semigroup of contractions $(T(t))_{t\geq0}$, before turning to analyse the spectrum of the generator $A$ of $(T(t))_{t\geq0}$.

\subsection{Existence of the Semigroup}

We start by recasting (\ref{WHWeq1}) into an abstract Cauchy problem in order to later apply the methods of non-uniform stability. Consider the Hilbert space $$X_0 = H^1(0,1)\times L^2(0,1)\times L^2(1,2) \times H^1(2,3) \times L^2(2,3)$$ and define
$$X= \{(u,v,w,\ut,\vt) \in X_0 : u(0) = \ut(3) = 0\}$$
endowed with the norm (and corresponding inner product) given by
$$\|(u,v,w,\ut,\vt)\|_X^2 = \|u'\|_{L^2}^2 + \|v\|_{L^2}^2 + \|w\|_{L^2}^2+\|\ut'\|_{L^2}^2+\|\vt\|_{L^2}^2$$ which is non-degenerate because the fundamental theorem of calculus applied in conjunction with the boundary conditions $u(0) = \ut(3) = 0$ implies that $\|u\|_{L^2}\lesssim \|u\|_{L^2}$ and $\|\ut\|_{L^2}\lesssim \|\ut'\|_{L^2}$. Here and in the rest of the article, the intervals for function spaces appearing as subscripts will often be omitted if they are clear from the context. Let $$X_1=X\cap [H^2(0,1)\times H^1(0,1) \times H^2(1,2) \times H^2(2,3)\times H^1(2,3)]$$ and define the operator $A$ on $X$ by $Ax = (v,u'',w'',\vt,\ut'')$ for $x=(u,v,w,\ut,\vt)$ in the domain
\begin{equation*}\begin{split}D(A) = \{(u,v,w,\ut,\vt) \in X_1 : & \ v(0) = \vt(3) = 0, u'(1) = w'(1), \\ & \ \ \ v(1) = w(1), \ut'(2) = w'(2), \vt(2) = w(2)\}. \end{split}\end{equation*}

\begin{lemma}\label{biglemma1} The following hold:
	\begin{enumerate}[(i)]
		\item $A$ is closed;
		\item $A$ is densely defined;
		\item $A$ is dissipative;
		\item $1-A$ is surjective.
	\end{enumerate}
\end{lemma}

\begin{proof}
	(i) Let $x_n = (u_n,v_n,w_n,\tilde{u}_n,\tilde{v}_n) \in D(A)$ be such that $$x_n\to x = (u,v,w,\ut,\vt), \ Ax_n = (v_n,u_n'',w_n'',\vt_n,\ut_n'') \to y = (f,g,h,\ft,\gt)$$ in $X$. Then $u_n$ converges to $u$ in $H^1(0,1)$ and $u_n''$ converges to $g$ in $L^2(0,1)$. Hence
	\begin{equation}\label{wkd}\int u\varphi'' = \lim_{n\to \infty}\int u_n\varphi'' = \lim_{n\to\infty}\int u_n'' \varphi = \int g\varphi, \ \ \  \varphi \in C_c^\infty(0,1),\end{equation} where the integral is taken over $((0,1),d\xi)$ so that $u \in H^2(0,1)$ and $u'' = g$. As $v_n$ converges to both $v$ and $f$ in $L^2(0,1)$, $v=f$. In particular, $v \in H^1(0,1)$. The same argument shows that $\ut \in H^2(2,3)$ with $\ut'' = \gt$ and $\vt = \ft \in H^1(2,3)$.
	
	Next, $w_n$ converges to $w$ and $w_n''$ to $h$ in $L^2(1,2)$. Standard Sobolev theory (see for example \cite[Page ~217]{Brezis}) ensures the existence of a constant $C$ such that $$\|\psi'\|_{L^2(1,2)} \leq \|\psi''\|_{L^2(1,2)} + C\|\psi\|_{L^2(1,2)} \ \ \ \psi \in H^2(1,2).$$ Hence, the sequence $w_n'$ is Cauchy and converges to some $H$ in $L^2(1,2)$. Using similar reasoning to that in (\ref{wkd}), we see that $w\in H^2(1,2)$ with $w' = H$ and $w'' = h$.
	
	To check that the coupling conditions for $x$ to be in the domain $D(A)$ are satisfied, it is enough to pass to a subsequence $x_{n_k}$ that converges pointwise a.e.\ and note the continuity of $u', v, w', w, \ut',\vt$. It follows that $Ax = y$.
	%As in (\ref{wkd}),
	
	(ii) Consider the subspace $X_1$ equipped with the $X$ norm, which is dense in $X$.
	%$$Z_1 = \{x = (u,v,w,\ut,\vt) \in Z : u,v,w,\ut,\vt \text{ are continuous}\}$$ which is dense in $Z$.
	The linear functional $\phi_1 : x = (u,v,w,\ut,\vt) \mapsto v(0)$ is unbounded on $X_1$, and hence
	$$X_2 = \Ker \phi_1 = \{(u,v,w,\ut,\vt) \in X_1 : v(0)=0 \}$$ is dense in $X_1$. Similarly,
	$$X_3 = \Ker \phi_2 = \{(u,v,w,\ut,\vt) \in X_2 : v(1)=w(1)\}$$ is dense in $X_2$ where $\phi_2$ is the unbounded linear functional on $X_2$ defined by $x \mapsto v(1)-w(1)$. Again, by considering the unbounded linear functional $\phi_3 : x \mapsto u'(1) - w'(1)$ on $X_3$, we see that $$X_4 = \Ker \phi_3 = \{(u,v,w,\ut,\vt) \in X_3 : u'(1)=w'(1)\}$$ is dense in $X_3$. The same argument can be repeated for the coupling and boundary conditions for $w,\ut,$ and $\vt$ to produce a decreasing finite chain of subspaces $$X \supset X_1 \supset X_2 \supset ... \supset D(A),$$ where each subspace is dense in the preceding one under the $X$ norm. Hence $A$ is densely defined.
	
	(iii) Let $x \in D(A)$. Assuming the appropriate intervals over which to take the $L^2$ inner products, we have, through integration by parts and the coupling and boundary conditions,
	\begin{align*}\langle Ax,x\rangle & = \langle v',u'\rangle_{L^2} + \langle u'',v\rangle_{L^2} + \langle w'',w\rangle_{L^2} + \langle \vt',\ut'\rangle_{L^2} + \langle \ut'',\vt\rangle_{L^2} \\ & = - \overline{\langle u'',v\rangle_{L^2}} + \langle u'',v\rangle_{L^2} - \langle w',w'\rangle_{L^2} - \overline{\langle \ut'',\vt\rangle_{L^2}} + \langle \ut'',\vt\rangle_{L^2}. \end{align*} Hence $$ \real{\langle Ax,x\rangle} = -\|w'\|_{L^2}^2 \leq 0,$$ showing that $A$ is dissipative.
	
	(iv) Though in the setting of this lemma, we only need to work with $1-A$, we perform a procedure here with $\lb-A$ for general $\lambda \neq 0$ in order to avoid repetition that otherwise would be inevitable in later sections. Note that we are closely following the proof of \cite[Theorem 3.1]{BPS1}.
	
	Let $x = \xt$ and $y = (f,g,h,\ft,\gt)$ be in $X$. Then the equation $(\lambda-A)x = y$ can be rewritten as the following system of boundary value problems: \begin{subequations}\begin{align} u'' & = \lambda^2 u-\lambda f -g, & \xi \in (0,1), \label{eqa}\\ v & = \lambda u-f, & \xi \in (0,1), \\  w'' & = \lambda w-h, & \xi \in (1,2), \label{eqc}\\ \ut'' & = \lambda^2\ut - \lambda\ft - \gt, & \xi \in (2,3), \label{eqd}\\ \vt & = \lambda \ut - \ft, & \xi \in (2,3),
		\\ u(0) = v(0) = 0, \ \ \ v(1) & = w(1), \ \ \ u'(1) = w'(1), \\ \ut(3) = \vt(3)=0, \ \ \ \vt(2) & = w(2), \ \ \ \ut'(2) = w'(2).
		\end{align}\end{subequations}
	Let
	$$U_\lambda(\xi) =  \frac{1}{\lambda}\int_{0}^\xi \sinh(\lambda(\xi-r))(\lambda f(r)+g(r))\ dr, \ \  \ \xi \in [0,1],$$ which has derivative $$U_\lambda'(\xi) =\int_{0}^\xi \cosh(\lambda(\xi-r))(\lambda f(r)+g(r))\ dr, \ \ \ \xi \in [0,1].$$ The differential equation (\ref{eqa}) with the boundary condition $u(0) = 0$ has the general solution 
	\begin{equation}\label{equ}
	u(\xi) = a(\lambda)\sinh(\lambda\xi) - U_\lambda(\xi), \ \ \ \xi \in [0,1],
	\end{equation} where $a(\lambda) \in \C$ is a parameter free to be varied. In particular,
	\begin{equation}\label{equ'}
	u'(\xi) = \lambda a(\lambda)\cosh(\lambda\xi) - U_\lambda'(\xi), \ \ \ \xi \in [0,1].
	\end{equation}. Clearly $u\in H^2(0,1)$ and hence $v\in H^1(0,1)$ with $v(0)=\lambda u(0)-f(0) =0$.
	
	Similarly, the general solution of (\ref{eqd}) with boundary condition $\ut(3)=0$ can be written as
	\begin{equation}\label{equt}
	\ut(\xi) = \at(\lambda)\sinh(\lambda(3-\xi)) + \tilde{U}_\lambda(\xi), \ \ \ \xi \in [2,3],
	\end{equation} where $\at(\lambda) \in \C$ can be varied freely and
	$$\tilde{U}_\lambda(\xi) =  \frac{1}{\lambda}\int_{\xi}^3 \sinh(\lambda(r-\xi))(\lambda\ft(r)+\gt(r))\ dr, \ \ \ \xi \in [2,3].$$ Thus
	\begin{equation}\label{equt'}
	\ut'(\xi) = -\lambda\at(\lambda)\cosh(\lambda(3-\xi)) + \tilde{U}_\lambda'(\xi), \ \ \ \xi \in [2,3],
	\end{equation} where
	$$\tilde{U}_\lb'(\xi) = -\int_{\xi}^3 \cosh(\lambda(r-\xi))(\lambda\ft(r)+\gt(r))\ dr, \ \ \ \xi \in [2,3].$$ Again, it follows that $\ut \in H^2(2,3)$ and $\vt \in H^1(2,3)$ with $\vt(3)=0$.
	
	In the same spirit, let
	$$W_\lambda(\xi) =  \frac{1}{\sqrt{\lambda}}\int_{1}^\xi \sinh(\sqrt{\lambda}(\xi-r))h(r)\ dr, \ \ \ \xi \in [1,2],$$ which has derivative
	$$W_\lambda'(\xi) =  \int_{1}^\xi \cosh(\sqrt{\lambda}(\xi-r))h(r)\ dr, \ \ \ \xi \in [1,2].$$ The general solution of (\ref{eqc}) can then be written as
	\begin{equation}\label{eqw}
	w(\xi) = b(\lambda)\cosh(\sqrt{\lambda}(\xi-1))+c(\lambda)\sinh(\sqrt{\lambda}(\xi-1)) - W_\lambda(\xi), \ \ \ \xi \in [1,2],
	\end{equation} where $b(\lambda), c(\lambda) \in \C$ are free parameters and in particular,
	\begin{equation}\label{eqw'}
	w'(\xi) = \sqrt{\lambda}b(\lambda)\sinh(\sqrt{\lambda}(\xi-1))+\sqrt{\lambda}c(\lambda)\cosh(\sqrt{\lambda}(\xi-1)) - W_\lambda'(\xi), \ \ \ \xi \in [1,2].
	\end{equation} It remains to choose specific constants $a(\lb),b(\lb),c(\lb)$ and $\at(\lb)$ in order to satisfy the coupling conditions. Using (\ref{equ}) and (\ref{eqw}), the requirement $\lambda u(1)-f(1) = v(1) = w(1)$ holds if and only if
	$$\lambda a(\lambda)\sinh(\lambda) - b(\lambda) = \lambda U_\lambda(1) + f(1).$$ Likewise, the conditions $u'(1) = w'(1)$, $\lambda \ut(2)-\ft(2) = w(2)$, and $\ut'(2) = w'(2)$ are equivalent to $$\lambda a(\lambda)\cosh(\lambda) -\sqrt{\lambda}c(\lambda) = U_\lambda'(1),$$ $$\lambda \tilde{a}(\lambda) \sinh(\lambda) - b(\lambda)\cosh(\sqrt{\lambda})-c(\lambda)\sinh(\sqrt{\lambda}) = -\lb\tilde{U}_\lambda(2) + \ft(2) - W_\lambda(2),$$ and $$ -\lambda \at(\lambda) \cosh(\lambda) -\sqrt{\lambda}b(\lambda)\sinh(\sqrt{\lambda}) - \sqrt{\lambda}c(\lambda)\cosh(\sqrt{\lambda}) = -\tilde{U}_\lambda'(2) - W_\lambda'(2)$$ respectively. These four equations can be written in matrix form as
	\begin{equation}\label{eqM}
	M_\lambda \cdot \begin{pmatrix} a(\lambda) \\ b(\lambda) \\ c(\lambda) \\ \at(\lambda) \end{pmatrix} = \mathbf{b},
	\end{equation} where \begin{equation}
	M_\lambda = \begin{pmatrix}
	\lambda\sinh(\lambda) & -1 & 0 & 0 \\ \lambda \cosh(\lambda) & 0 & -\sqrt{\lambda} & 0 \\ 0 & -\cosh(\sqrt{\lambda}) & -\sinh(\sqrt{\lambda}) & \lambda\sinh(\lambda) \\ 0 & \sqrt{\lambda}\sinh(\sqrt{\lambda}) & \sqrt{\lambda}\cosh(\sqrt{\lambda}) & \lambda\cosh(\lambda)
	\end{pmatrix}
	\end{equation} and \begin{equation}\mathbf{b} = \begin{pmatrix}\lambda U_\lambda(1) + f(1) \\ U_\lambda'(1) \\ -\lb\tilde{U}_\lambda(2) +\ft(2) - W_\lambda(2) \\ \tilde{U}_\lambda'(2) +W_\lambda'(2)
	\end{pmatrix}. \end{equation} Thus, (\ref{eqM}) has a solution for any given $y = (f,g,h,\ft,\gt)$ in $X$ if and only if $$\det M_\lambda = -\lambda^2[2\sqrt{\lambda}\cosh(\sqrt{\lambda})\cosh(\lambda)\sinh(\lambda)+\sinh(\sqrt{\lambda})(\lambda\sinh^2(\lambda)+\cosh^2(\lambda))]$$ is non-zero. For $\lambda = 1$,
	$$\det M_1 = -\sinh(1)[4\cosh^2(1) - 1]\neq 0,$$ proving (4).%= -[2\cosh^2(1)\sinh(1) +\sinh^3(1) + \sinh(1)\cosh^2(1)]  Hence $1-A$ is surjective.
\end{proof}

All the dirty work has now been done (ahead of time). The following theorem follows immediately from Lemma \ref{biglemma1} and the Lumer-Phillips theorem.

\begin{theorem}\label{wpthm}
	$A$ generates a contractive $C_0$-semigroup $T(t)$ on $X$.
\end{theorem}

\subsection{Spectrum of the Generator}

From Theorem \ref{wpthm} and the Hille-Yosida theorem, we know that $\sigma(A)$ is contained in the closed left half-plane. However, we can say more about the spectrum.

\begin{theorem}\label{WHWgenspec}
	The spectrum of $A$ consists of isolated eigenvalues and is given by
	$$\sigma(A) = \{\lambda \in \C_- : \det M_\lambda = 0\}.$$ In particular, $\sigma(A) \cap i\R = \emptyset$.
\end{theorem}

We will need the following lemma in order to prove the theorem above.

\begin{lemma}\label{compres}
	If $\lambda \in \rho(A)$, then $R(\lambda,A)$ is a compact operator.
\end{lemma}

\begin{proof}
	Let $\lb \in \rho(A)$. Then $\lambda -A$ is a bijective bounded (and in particular, closed) linear map from $D(A)$ endowed with the graph norm onto $X$. Hence the inverse map $R(\lb,A)$ maps $X$ isomorphically onto $(D(A),\|\cdot\|_{D(A)})$. Since
	\begin{align*}\|(u,v,w,\ut,\vt)\|_{D(A)} & = \|(u,v,w,\ut,\vt)\|_X+ \|(v,u'',w'',\vt,\ut'')\|_X \\ & \lesssim  \|u'\|_{L^2} + \|v\|_{L^2} + \|w\|_{L^2} + \|\ut'\|_{L^2} + \|\vt\|_{L^2} \\ & \qquad + \|v'\|_{L^2} + \|u''\|_{L^2} + \|w''\|_{L^2} + \|\vt'\|_{L^2} + \|\ut''\|_{L^2}, \end{align*} it follows that $(D(A),\|\cdot\|_{D(A)})$ embeds continuously into $$H^2(0,1)\times H^1(0,1) \times H^2(1,2) \times H^2(2,3)\times H^1(2,3)$$ endowed with its natural norm (see \cite[Page 217]{Brezis}). This space in turn embeds compactly into $X$ by the Rellich-Kondrachov theorem of Sobolev theory. Stringing together these embeddings, $R(\lb,A)$ is a compact operator on $X$.
\end{proof}

\begin{proof}[Proof of Theorem \ref{WHWgenspec}]
	We first show that not only is $\lb-A$ surjective as shown in Lemma \ref{biglemma1} whenever $\det M_\lambda \neq 0$, it is also injective. Indeed, suppose $(\lb-A)x = 0$. Then, $x$ is obtained in the same way as in the proof of Lemma \ref{biglemma1}(4) with $\mathbf{b} =0$ in (\ref{eqM}). As $\det M_\lb \neq 0$, we get that $x = 0$. Hence $\lb-A$ is closed and bijective, so has bounded inverse by the closed graph theorem. In particular, $1 \in \rho(A)$ and so the resolvent is non-empty.
	
	The spectral theorem for compact operators used in conjunction with Lemma \ref{compres} implies that the spectrum of $R(1,A)$ consists only of eigenvalues of finite multiplicity with the only possible accumulation point being the origin. By the spectral mapping theorem for the resolvent,
	$$\sigma(A) = \{1-\nu^{-1} : \nu \in \sigma(R(1,A))\setminus\{0\}\}$$ and furthermore, a simple calculation shows that if $\nu$ is an eigenvalue of $R(1,A)$, then $1- \nu^{-1}$ is an eigenvalue of $A$. Hence $\sigma(A)$ consists only of eigenvalues of finite multiplicity with the only possible accumulation point being at infinity. Thus, $\lb \in \sigma(A)$ if and only if $\det M_\lb = 0$.
	
	To show the final statement, suppose that $s\in \R$ with $s\neq 0$ and that $x = (u,v,w,\ut,\vt) \in \Ker (is-A)$. From the proof of Lemma \ref{biglemma1}(3), we have
	\begin{equation}\label{w=0}0 = \real{\langle (is-A)x,x\rangle} = -\real{\langle Ax,x\rangle} =\|w'\|_{L^2}.\end{equation} Thus $w = (is)^{-1}(w')' =0$. As in the proof for Lemma \ref{biglemma1}(4), we have
	$$u(\xi) = a(is)\sinh(is\xi), \ \ \ v(\xi) = is\ u(\xi), \ \ \ \xi \in [0,1].$$ The coupling conditions imply that $u'(1) = v(1) =0$. Thus, $$is\ a(is)\cosh(is) = is\ a(is)\sinh(is) =0,$$ implying that $a(is) =0$. Similarly, $\at(is) =0$ so that $x=0$.
	
	Consider now the case $s=0$. Rewriting $Ax =0$ into component differential equations, we get that $u''=0$ and $v=0$ as well as $w=0'$ as in (\ref{w=0}). As $u'(1)=w'(1)$ and $u'$ is constant, $u' = 0$ and hence $u(0)=0$ implies that $u=0$. Similarly, $v(1)=0$ implies that $w=0$. The same is true for $\ut$ and $\vt.$ It follows that $\sigma(A) \cap i\R = \emptyset$.
\end{proof}

\section{Resolvent Estimates}\label{resolventestimates}

We turn now to obtaining an upper bound on the growth of $\|R(is,A)\|$ as $|s| \to \infty$ which will allow us to deduce a quantitative estimate on the rate of energy decay in the next section.

\begin{theorem}\label{mainthm}
	We have $\|R(is,A)\| = O(|s|^{1/2})$ as $|s| \to \infty$.
\end{theorem}

To prove this theorem, we will need explicit forms for the $a(\lambda),b(\lambda),c(\lambda),\at(\lambda)$ found in the proof of Lemma \ref{biglemma1}(4) for the case where $\lambda = is$ and to this end, we invert $M_\lambda$ to get that \begin{equation}
(\det M_\lambda)^{-1} C^{T} \mathbf{b} = \begin{pmatrix}a(\lambda) \\ b(\lambda) \\ c(\lambda) \\ \at(\lambda)\end{pmatrix},
\end{equation} where $C$ is the cofactor matrix of $M_\lb$. First, we rewrite $\det M_\lambda$ and define two terms which are ubiquitous in this section:
\begin{align*}
\det M_\lambda & = -\lambda^2[2\sqrt{\lambda}\cosh(\sqrt{\lambda})\cosh(\lambda)\sinh(\lambda)+\sinh(\sqrt{\lambda})(\lambda\sinh^2(\lambda)+\cosh^2(\lambda)] \\
& = \frac{\lambda^2}{2}[-e^{\sqrt{\lambda}}(\lambda \sinh^2(\lambda)+2\sqrt{\lambda}\cosh(\lambda)\sinh(\lambda)+\cosh^2(\lambda)) \\ & \ \ \ \ \ \ \ \ \ \ \ + e^{-\sqrt{\lambda}}(\lambda \sinh^2(\lambda)-2\sqrt{\lambda}\cosh(\lambda)\sinh(\lambda)+\cosh^2(\lambda))] \\ & = \lambda^2[-e^{\sqrt{\lambda}}T_+^2(\lambda) + e^{-\sqrt{\lambda}}T_-^2(\lambda)],
\end{align*} where $$T_+(\lambda) =\frac{1}{2}[\cosh(\lambda)+\sqrt{\lambda}\sinh(\lambda)], \ \ \ T_-(\lambda) = \frac{1}{2}[\cosh(\lambda)- \sqrt{\lambda}\sinh(\lambda)].$$

The functions $T_+$ and $T_-$ are useful because they obey convenient lower bounds on the one hand, and appear many times in the entries of $C = \{c_{ij}\}_{i,j}$ on the other hand. As an example of this, $c_{11}$ is explicitly computed and stated here:
\begin{align*}
c_{11} & = -\lb^{3/2}[\cosh(\sqrt{\lambda})\cosh(\lambda) +\sqrt{\lambda}\sinh(\sqrt{\lambda})\sinh(\lambda)] \\ & = -\lb^{3/2}[e^{\sqrt{\lambda}}T_+(\lb) + e^{-\sqrt{\lambda}}T_-(\lb)].
\end{align*}
The expressions for the other entries can be found in the appendix.

We will also need the following two lemmas, the first of which is proved in \cite[Lemma 3.3]{BPS1} (over the interval $[-1,0]$ rather than $[0,1]$ or $[2,3]$ as we have here).

\begin{lemma}\label{1U}
	There exists a constant $C\geq 0$ such that, for all $f \in H^1(0,1), g\in L^2(0,1), \ft \in H^1(2,3), \gt \in L^2(2,3),$ and $\lambda\in i\R$,
	\begin{align*}
	\left|\int_{0}^\xi \sinh(\lambda(\xi-r))(\lambda f(r)+g(r))dr\right| & \leq C\|f\|_{H^1}+\|g\|_{L^2}, \ \ \ \xi \in [0,1], \\
	\left|\int_{0}^\xi \cosh(\lambda(\xi-r))(\lambda f(r)+g(r))dr\right| & \leq C\|f\|_{H^1}+\|g\|_{L^2}, \ \ \ \xi \in [0,1], \\
	\left|\int_{\xi}^3 \sinh(\lambda(r-\xi))(\lambda\ft(r)+\gt(r))dr\right| & \leq C\|\ft\|_{H^1}+\|\gt\|_{L^2}, \ \ \ \xi \in [2,3], \\
	\left|\int_{\xi}^3 \cosh(\lambda(r-\xi))(\lambda\ft(r)+\gt(r))dr\right| & \leq C\|\ft\|_{H^1}+\|\gt\|_{L^2}, \ \ \ \xi \in [2,3].
	\end{align*}
\end{lemma}

\begin{lemma}\label{pm}
	For $\lambda\in i\R$ with $|\lambda| \geq \left(\frac{1}{\sqrt{2}}+1\right)^2$, we have
	$$|T_+(\lambda)|,|\ T_-(\lambda)| \geq 1/4.$$%\leq \sqrt{|\lambda|} + 1.$$
\end{lemma}

\begin{proof}
	We prove this for $T_+(\lambda)$ where $\lambda = is$ with $s\in\R$ and note that $2T_+(is) = \cos(s)+i\sqrt{is}\sin(s)$. Explicit calculation yields
	\begin{align*}4|T_+(\lambda)|^2 & = \left|\frac{1}{\sqrt{is}}(\sqrt{is}\cos(s)-s\sin(s))\right|^2 \\ & = \frac{1}{|s|}\left(\frac{|s|}{2}\cos^2(s) + \left(\sqrt{\frac{|s|}{2}}\cos(s)-s\sin(s)\right)^2\right),\end{align*} as $\real{\sqrt{\lb}} \geq 0$ for all $\lb\in\C$ since we have taken the branch cut of the square root along the negative real axis. In the case where $\cos^2(s) \geq 1/2$, it follows that $4|T_+(\lambda)|^2 \geq 1/4$. However, if $\cos^2(s) < 1/2$, then $|\sin(s)|^2\geq1/2$, so that
	$$2|T_+(\lambda)| \geq |i\sqrt{is}\sin(s)| - |\cos(s)| \geq \sqrt{\frac{|s|}{2}} - \frac{1}{\sqrt{2}} \geq \frac{1}{2}$$ whenever $|\lambda| = |s| \geq \left(\frac{1}{\sqrt{2}}+1\right)^2$. The case for $T_-(\lambda)$ is similar.
\end{proof}

\begin{proof}[Proof of Theorem \ref{mainthm}]
	Let $\lambda =is$ for $s\in \R$ and let $y= (f,g,h,\ft,\gt) \in Z$, further defining $x = (u,v,w,\ut,\vt) \in D(A)$ by $x = R(\lambda,A)y.$ As $v = \lambda u -f$ and $\vt = \lambda\ut$, we have that
	$$ \|x\| \lesssim \|\lambda u\|_{L^2} + \|f\|_{L^2}+ \|u'\|_{L^2} +\|w\|_{L^2} +\|\lambda \ut\|_{L^2} +\|\ft\|_{L^2}+ \|\ut'\|_{L^2}.$$ Thus the result will follow once we have established that each of the summands in the above equation are bounded by $C\sqrt{|\lb|}\|y\|$ for $|s| \geq N$, where $C,N >0$ are constants independent of $y$.
	
	Consider $u$ given by (\ref{equ}). By Lemma \ref{1U}, it is enough to consider $|\lambda a(\lambda)|$ in order to estimate $\|\lambda u\|_{L^2}$ and $\|u'\|_{L^2}$. Now
	\begin{equation}\label{c11}\lambda a(\lambda) = \frac{\lambda}{\det M_\lambda}(c_{11}b_1 + c_{21}b_2+c_{31}b_3 + c_{41}b_4)\end{equation} where $b_i$ are the components of the vector $\mathbf{b}$ in (\ref{eqM}). We consider each of these terms. Note that by lemma \ref{1U}, the only terms in the components of $\mathbf{b}$ that are not automatically bounded by some constant multiple of $\|y\|$ are $W_\lambda(2)$ and $W_\lambda'(2)$. Looking at the first term in (\ref{c11}),
	
	$$\left|\frac{\lambda}{\det M_\lb}c_{11}\right| = \sqrt{|\lambda|}\left|\frac{\esl T_+(\lambda) + \nesl T_-(\lambda)}{-\esl T_+(\lambda)^2 + \nesl T_-(\lambda)^2}\right| \lesssim \sqrt{|\lb|}|T_{+}(\lb)|^{-1} \lesssim \sqrt{|\lb|},$$ since $\real{\sqrt{\lb}}>0$ for $\lb = \in i\R\setminus\{0\}$ as before, so that $\esl$ dominates $\nesl$. Thus,
	$$\left|\frac{\lb}{\det M_\lb}c_{11}b_1\right|\lesssim \sqrt{|\lb|}\|y\|,$$ where the implicit constant is independent of $\lb$ and $y$.
	
	Likewise,
	$$\left|\frac{\lambda}{\det M_\lb}c_{21}\right|= \left|\frac{\esl T_+(\lambda) - \nesl T_-(\lambda)}{-\esl T_+(\lambda)^2 + \nesl T_-(\lambda)^2}\right| \lesssim 1,$$ so that $$\left|\frac{\lb}{\det M_\lb}c_{21}b_2\right|\lesssim \|y\|.$$
	
	Noting that $|\cosh(\lambda)|,|\sinh(\lambda)| \leq 1$, a similar argument shows that the remaining terms in (\ref{c11}) that do not include $W_\lb(2)$ and $W_\lambda(2)'$ are bounded by a constant times $\sqrt{|\lb|}\|y\|$. Consider now
	\begin{align*}\left|\frac{\lambda}{\det M_\lb}c_{31}W_\lb(2)\right| & \leq \int_1^2 \left|\frac{\sinh(\sqrt{\lambda}(2-r))h(r)}{-\esl T_+(\lb)^2+\nesl T_-(\lb)^2}\right|dr \\ & \leq \frac{1}{2}\int_1^2 \left|\frac{e^{\sqrt{\lb}(2-r)}-e^{-\sqrt{\lb}(2-r)}}{-\esl T_+(\lb)^2 +\nesl T_-(\lb)^2}\right||h(r)| dr \\ & \lesssim |T_+(\lb)|^{-2}\|h\|_{L^2} \lesssim \|h\|_{L^2} \end{align*} where the inequality in the final line is justified as before noting that $2-r \in[0,1]$. Similarly,
	$$\left|\frac{\lambda}{\det M_\lb}c_{41}W_\lb'(2)\right| \lesssim \sqrt{|\lb|}\|h\|_{L^2}.$$ These inequalities combined with (\ref{c11}) imply that
	$$|\lambda a(\lb)| \lesssim \sqrt{|\lb|}\|y\| \ \ \ (|\lb|\geq N)$$ for some constant $N>0$ independent of $y$ and in particular,
	$$\|\lb u\|_{L^2}, \ \|u'\|_{L^2} \lesssim \|y\| \ \ \ (|\lb| \geq N).$$ The same arguments show that this also holds for $\|\lb \ut\|_{L^2}$ and $\|\ut'\|_{L^2}$.
	
	We must now estimate $w$ given by (\ref{eqw}), noting that
	$$b(\lb) = \frac{1}{\det M_\lb}(c_{12}b_1+c_{22}b_2+c_{32}b_2+c_{42}b_4)$$ and $$c(\lb) = \frac{1}{\det M_\lb}(c_{13}b_1+c_{23}b_2+c_{33}b_2+c_{43}b_4).$$ The trick to estimating $w$ is to group the terms together in a specific way. As before, the only terms in the components of $\mathbf{b}$ which are not bounded by $\|y\|$ are $W_\lambda(2)$ and $W_\lb'(2)$. Hence it is enough to estimate the moduli of
	\begin{align*}
	w_1(\xi) & = \frac{1}{\det M_\lb}[c_{12}\cosh(\sqrt{\lb}(\xi-1))+c_{13}\sinh(\sqrt{\lb}(\xi-1))],\\
	w_2(\xi) & = \frac{1}{\det M_\lb}[c_{22}\cosh(\sqrt{\lb}(\xi-1))+c_{23}\sinh(\sqrt{\lb}(\xi-1))],\\
	w_3(\xi) & = \frac{1}{\det M_\lb}[c_{32}\cosh(\sqrt{\lb}(\xi-1))+c_{33}\sinh(\sqrt{\lb}(\xi-1))],\\
	\omega_4(\xi) & = \frac{1}{\det M_\lb}[c_{42}\cosh(\sqrt{\lb}(\xi-1))+c_{43}\sinh(\sqrt{\lb}(\xi-1))],\end{align*} and \begin{equation*}\begin{split}
	w_5(\xi) & = \frac{1}{\det M_\lb}\Big[\big(-c_{32}W_\lb(2)+c_{42}W_\lb'(2)\big)\cosh(\sqrt{\lb}(\xi-1))\\ & \qquad \ \ + \big(-c_{33}W_\lb(2)+c_{43}W_\lb'(2)\big)\sinh(\sqrt{\lb}(\xi-1)) - \det M_\lb W_\lb(\xi)\Big],\end{split}
	\end{equation*} where $\xi \in[1,2]$, since the sum of the $w_i$ is equal to the $w$ after removing the terms of $\mathbf{b}$ that do not include $W_\lb(2)$ and $W_\lb'(2)$. These removed terms can be shown to obey the desired estimates using the previous method.
	
	Plugging in the appropriate values gives
	\begin{align*}w_1 & = \frac{\lb^2\cosh(\lb)}{\det M_\lb}\Big[\esl T_+(\lb)\big(\cosh(\sqrt{\lb}(\xi-1)) - \sinh(\sqrt{\lb}(\xi-1))\big)\\ & \quad \quad \quad \quad \quad - \nesl T_-(\lb) \big(\cosh(\sqrt{\lb}(\xi-1)) +\sinh(\sqrt{\lb}(\xi-1))\big)\Big] \\ 
	& = \frac{\lb^2\cosh(\lb)}{\det M_\lb} \left(\esl T_+(\lb)e^{-\sqrt{\lb}(\xi-1)} - \nesl T_-(\lb)e^{\sqrt{\lb}(\xi-1)}\right)\\
	& = \cosh(\lb)\frac{e^{\sqrt{\lb}(2-\xi)}T_+(\lb) - e^{-\sqrt{\lb}(2-\xi)}T_-(\lb)}{-\esl T_+(\lb)^2+\nesl T_-(\lb)^2}.\end{align*}
	Since $2-\xi \in [0,1]$, as in the case for $u$, $$|w_1(\xi)| \lesssim|T_\pm(\lb)|^{-1} \lesssim 1,$$ where the sign of $\pm$ is determined by that of $s$ and the growth bound is independent of $\xi$. Likewise, $|w_2(\xi)| \lesssim 1$ with the bound independent of $\xi$. Next we have that
	\begin{align*}
	w_3 & = \cosh\lb \frac{\sqrt\lb \sinh(\lb)\cosh(\sqrt\lb(\xi-1))+\cosh(\lb)\sinh(\sqrt\lb(\xi-1))}{-\esl T_+(\lb)^2 + \nesl T_-(\lb)^2} \\
	& = \cosh(\lb)\frac{e^{\sqrt\lb(\xi-1)}T_+(\lb) - e^{-\sqrt\lb(\xi-1)}T_-(\lb)}{-\esl T_+(\lb)^2 + \nesl T_-(\lb)^2}.
	\end{align*} Since $\xi -1 \in[0,1]$, the previous argument again shows that $|w_3(\xi)| \lesssim 1$ with the bound independent of $\xi$. The same holds for $w_4$. Thus, it remains to estimate $w_5$ and after some simple manipulation, we can rewrite this as \begin{equation}\label{wt}w_5 = \frac{\lambda^2}{\det M_\lb}\left[\frac{\cosh^2(\lambda)}{\sqrt{\lb}}\Omega_1(\xi) +\slb\sinh^2(\lb) \Omega_2(\xi) + \cosh(\lambda)\sinh(\lambda)\Omega_3(\xi)\right],\end{equation} where
	\begin{align*}\Omega_1(\xi) & = -\int_1^2 \sinh(\slb(2-r))\sinh(\slb(\xi-1))h(r)\ dr \\ & \ \ \ \ \ \ \ \ \ \ \ \ + \int_1^\xi \sinh(\slb(\xi-1))\sinh(\slb)h(r)\ dr \\ & = \frac{1}{2}\Big[-\int_\xi^2 \cosh(\slb(1+\xi-r))h(r)\ dr \\ & \ \ \ \ \ \ \ \ \ \ \ \ + \int_1^2 \cosh(\slb(3-\xi-r))h(r)\ dr \\ & \ \ \ \ \ \ \ \ \ \ \ \ \ \ \ \ \ \ \ \ \ \ \ \  - \int_1^\xi \cosh(\slb(\xi-r-1))h(r)\ dr \Big],\end{align*}
	with the second equality following from the use of identities such as \begin{equation*}\begin{split} 2\sinh(\slb(2-r))& \sinh(\slb(\xi-1))\\ & =  \cosh(\slb(2-r)+\slb(\xi-1)) \\ & \quad\quad  -\cosh(\slb(2-r) -\slb(\xi-1)),\end{split}\end{equation*}
	\begin{align*}\Omega_2(\xi) & = -\int_1^2 \cosh(\slb(2-r))\cosh(\slb(\xi-1))h(r)\ dr \\ & \ \ \ \ \ \ \ \ \ \ \ \ + \int_1^\xi \sinh(\slb(\xi-r))\sinh(\slb)h(r)\ dr \\ & = -\frac{1}{2}\Big[\int_\xi^2 \cosh(\slb(1+\xi-r))h(r)\ dr \\ & \ \ \ \ \ \ \ \ \ \ \ \ + \int_1^2 \cosh(\slb(3-\xi-r))h(r)\ dr \\ & \ \ \ \ \ \ \ \ \ \ \ \ \ \ \ \ \ \ \ \ \ \ \ \  + \int_1^\xi \cosh(\slb(\xi-r-1))h(r)\ dr \Big],\end{align*} and 
	\begin{align*}\Omega_3(\xi) & = 2\int_1^\xi \sinh(\slb(\xi-r))\cosh(\slb)h(r)\ dr \\ & \ \ \ \ \ \ -\int_1^2 \sinh(\slb(2-r))\cosh(\slb(\xi-1))h(r)\ dr \\ & \ \ \ \ \ \ \ \ \ \ \ \ -\int_1^2 \cosh(\slb(2-r))\sinh(\slb(\xi-1))h(r)\ dr \\ & = -\int_1^\xi \sinh(\slb(1-\xi+r))h(r)\ dr  -\int_\xi^2 \sinh(\slb(1+\xi-r))h(r)\ dr.\end{align*}
	
	Note that for $\xi \in [1,2]$, $|1+\xi-r|\leq 1$ whenever $r\in[\xi,2]$, and $|3-\xi-r|\leq 1$ whenever $r\in[1,2]$, and $|\xi-r-1|\leq 1$ whenever $r\in[1,\xi]$. Thus by pulling the factor of $\big(-\esl T_+(\lb)^2 + \nesl T_-(\lb)^2\big)^{-1}$ into the integrand of $\Omega_1$, we see that
	$$\left|\frac{\lb^2}{\det M_\lb} \frac{\cosh^2(\lb)}{\slb}\Omega_1(\xi)\right| \lesssim \frac{1}{\slb}\|h\|_{L^2}.$$ Arguing similarly for $\Omega_2$ and $\Omega_3$, we get from (\ref{wt}) that
	$$|w_5| \lesssim \left(\frac{1}{\sqrt{|\lb|}} + \sqrt{|\lb|} + 1\right)\|h\|_{L^2} \lesssim \sqrt{|\lb|}\|h\|_{L^2},$$ with the implicit constant independent of $\xi$. It follows that
	$$\|w\|_{L^2} \lesssim \|y\|, \ \ \ |\lb|\geq N,$$ for some constant $N>0$ independent of $y$ and, in particular,
	$$\|x\| \lesssim |\lb|^{1/2}\|y\|,$$ with the implicit constant independent of the specific $y$ and $x$.
\end{proof}

\section{Optimal Energy Decay for Classical Solutions}

For the reader's convenience, we state below the version of the Borichev-Tomilov theorem used in \cite{BPS1}.

\begin{theorem}[{\cite[Theorem 4.1]{BPS1}}]\label{boto}
Let $(T(t))_{t\geq0}$ be a bounded $C_0$-semigroup on a Hilbert space $X$ with generator $A$ such that $\sigma(A)\cap i\R = \emptyset$.	Then for any $\alpha >0$, the following are equivalent:
\begin{enumerate}[(i)]
	\item $\|R(is,A)\| = O(|s|^\alpha)$ as $|s|\to\infty$;
	\item $\|T(t)A^{-1}\| = O(t^{-1/\alpha})$ as $t\to \infty$;
	\item $\|T(t)x\| = o(t^{-1/\alpha})$ as $t\to \infty$ for all $x\in D(A)$.
\end{enumerate}
\end{theorem}

Using the abstract but powerful tool above, we can convert the resolvent estimate in Theorem \ref{mainthm} into a rate of energy decay of classical solutions of (\ref{WHWeq1}), deriving the main result of the article. The rate itself will follow easily from Theorem \ref{boto} as we shall soon see, but optimality will require a little more work.

\begin{theorem}\label{opt}
	If $x \in D(A)$, then $E_x(t) = o(t^{-4})$ as $t\to \infty$. Moreover, this rate is optimal in the sense that, given any positive function $r$ satisfying $r(t) = o(t^{-4})$ as $t \to \infty$, there exists $x \in D(A)$ such that $E_x(t) \neq o(r(t))$ as $t \to \infty$.
\end{theorem}

Before we begin the proof, we state the following summary proposition needed to show optimality. What is stated below is more or less a collection of results from \cite{BPS1}. %two results needed to show optimality. The first result is from \cite{BaDu}, the precursor paper to \cite{BoTo}.

%\begin{proposition}[{\cite[Proposition 1.3]{BaDu}}]\label{neccon}
%	Let $(T(t))_{t\geq 0}$ be a bounded $C_0$-semigroup on a Banach space $X$ with generator $A$ such that $0 \in \rho(A)$. Let $$m(t)=\sup_{s\geq t}\|T(t)A^{-1}\|$$ and assume that
%	$$\lim_{t\to \infty}m(t) =0.$$ Then $\sigma(A)\cap i\R = \emptyset$ and
%	$$\sup_{-s\leq \tau\leq s}\|R(i\tau,A)\| \lesssim 1+m_{*}^{-1}\left(\frac{1}{2(s+1)}\right),$$ where $m_*^{-1}$ is a right inverse of the non-increasing function $m$, mapping $(0,m(0)]$ onto $[0,\infty)$.
%\end{proposition}

\begin{proposition}\label{halfprop}
	Let $B$ be the generator of the $C_0$-semigroup $S(t)$ on the Hilbert space $$Z_* = \{(u,v,w) \in H^1(0,1)\times L^2(0,1) \times L^2(1,3/2) : u(0)=0\}$$ that solves the following well-posed problem:
	\begin{equation}\label{WHWeq2}
	\begin{cases}\begin{aligned} & u_{tt}(\xi,t) = u_{\xi\xi}(\xi,t), & \xi \in (0,1), \ & t>0,\\
	& w_t(\xi,t) = w_{\xi\xi}(\xi,t), & \xi \in (1,3/2), \ & t>0, \\
	& u(0,t) = w(3/2,t) = 0, & & t>0, \\
	& u_t(1,t) = w(1,t),\ \ \ u_\xi(1,t) = w_\xi(1,t), & & t>0, \\
	& u(\xi,0) = u(\xi), \ \ \ u_t(\xi,0) = v(\xi), & \xi \in (0,1), \ & \\
	& w(\xi,0) = w(\xi), & \xi \in (1,3/2). \ &
	\end{aligned}\end{cases}\end{equation}
	Then
	\begin{equation}\label{limsup}\limsup_{|s|\to\infty}|s|^{-1/2}\|R(is,B)\| >0.\end{equation} In particular, for any positive function $r$ satisfying $r=o(t^{-4})$ as $t\to \infty$, there exists \begin{align*}x_* \in D(B) = \{(u,v,w) \in H^2(0,1)&\times H^1(0,1)\times H^2(1,3/2) \\ & : u(0)=v(0)=w(3/2) = 0, \\ & \ \ \ \ \ \ v(1)=w(1), \ u'(1)=w'(1)\}\end{align*} such that $$E_{x_*}(t) = \int_0^1 |u'(\xi,t)|^2 + |v(\xi,t)|^2 \ d\xi + \int_1^{3/2} |w(\xi,t)|^2 \ d\xi \neq o(r(t))$$ as $t\to \infty$.
\end{proposition}

\begin{proof}
	After a rescaling of the heat component by a factor of $2$, this is the same problem as what is studied in \cite[Section 5]{BPS1}, namely the coupled wave-heat equation that leads to the optimal resolvent bound in \cite[Theorem 3.1]{BPS1}, but with Dirichlet wave condition. The problem is well-posed, therefore, and the same resolvent estimates hold up to a constant and they remain optimal in the sense of (\ref{limsup}). This is again proved in the same exact way as \cite[Theorem 3.4]{BPS1} by using the argument found there based on Rouch\'e's theorem. Note that in this case, however, $\sigma(B)\cap i\R = \emptyset$.
	
	The final part of the proposition follows from (\ref{limsup}) and is proved along the lines of \cite[Remark 4(a)]{BPS1}. We flesh that remark out here. Assume for a contradiction that there exists a positive function $r$ satisfying $r(t)=o(t^{-2})$ as $t\to \infty$ such that for all $x \in D(B)$ $\|S(t)x\| = o(r(t))$. Without loss of generality, $r$ is non-increasing since we can replace $r$ with $r_1(t) = \sup_{t\leq \tau}r(\tau)$ which also satisfies $r_1(t) = o(t^{-2})$ and $\|S(t)x\|=o(r_1(t))$ for all $x\in D(B)$. Then for all $y \in X$, there exists a constant $C_y$ such that
	$$r(t)^{-1}\|S(t)R(1,B)y\| \leq C_y, \ \ \ t\geq 0.$$ Hence by the uniform boundedness principle, there exists $C>0$ independent of $y$ such that
	$$\|S(t)R(1,B)\| \leq Cr(t).$$ In particular, $m(t) \leq Cr(t)$ where
	$m(t)=\sup_{t\leq \tau} \|S(\tau)R(1,B)\|.$ By \cite[Proposition 1.3]{BaDu},
	$$\|R(is,B)\| \lesssim 1+ m_{*}^{-1}\left(\frac{1}{2(|s|+1)}\right), \ \ \ s \in \R,$$ where $m_*^{-1}$ is a right inverse of the function $m$, mapping $(0,m(0)]$ onto $[0,\infty)$. This contradicts (\ref{limsup}) if $|s|^{-1/2}m_{*}^{-1}\left(\frac{1}{2(|s|+1)}\right)\to 0$ as $|s| \to \infty$, which we now show.
	
	Notice first that because $t^2Cr(t) \to 0$ as $t\to \infty$ and $(Cr)_*^{-1}(|s|)\to \infty$ as $|s|\to 0$, we have that $(Cr)_*^{-1}(|s|)^2|s| \to 0$ as $|s| \to 0$, where $(Cr)_*^{-1}$ is a right inverse of the function $Cr$, mapping $(0,Cr(0)]$ onto $[0,\infty)$.
	
	Hence
	$$(Cr)_*^{-1}\left(\frac{1}{2(|s|+1)}\right)^2\frac{1}{2(|s|+1)}\to 0$$ as $|s| \to \infty$. But since $m\leq Cr$ and both functions are non-increasing, it follows that $m_*^{-1} \leq (Cr)_*^{-1}$ on the interval $(0,m(0)]$ and we are done. %(see the appendix if you are not convinced).
	%	The problem is easily seen to be well-posed by appropriately modifying the arguments presented in Subsection 5.1.1 or in \cite{BPS1}. Replacing $F$ and $G$ in the proof of \cite[Theorem 3.4]{BPS1} by $\tanh(\lb)$ and $\tanh(\slb/2)/\slb$ respectively and adjusting the values of the poles and zeros accordingly yields (\ref{limsup}). For the final statement, see \cite[Remark 4.3(a)]{BPS1}.
\end{proof}

\begin{remark}
	In the above proof, we can alternatively prove the simpler optimality statement that $$\|S(t)R(1,B)\| \geq ct^{-2}, \ \ \ t\ge1,$$ for some constant $c>0$ by combining \cite[Theorem 4.4.14]{ABHN} with the fact that the specific $\lb_n^\pm$ in \cite[Theorem 3.4]{BPS1} are evenly spaced.
\end{remark}

We finally prove the decay rate in Theorem \ref{opt} using Theorem \ref{boto} as promised and its optimality by showing that the system (\ref{WHWeq2}) is effectively contained within (\ref{WHWeq1}).

\begin{proof}[Proof of Theorem \ref{opt}]
	By Theorem \ref{boto}, we have that
	$$E_x(t) = \frac{1}{2}\|T(t)x\|^2 = o(t^{-4})$$ as $t \to \infty$ for any $x \in D(A)$ since Theorem \ref{mainthm} gives us the rate $\|R(is,A)\| = O(|s|^{1/2})$ as $s \to \infty$.
	
	To show optimality, assume that there exists a positive function $r$ satisfying $r=o(t^{-4})$ as $t\to \infty$. Proposition \ref{halfprop} produces an $x_* \in D(B) \subset Z_*$ for which $E_{x_*}(t) \neq o(r(t))$ as $t \to \infty$. Define $\tilde{x}:[0,3]\to\C$ by
	$$\tilde{x}(\xi) = \begin{cases} x_*(\xi), & \xi \in [0,3/2], \\ -x_*(3-\xi), & \xi \in (3/2,3],
	\end{cases}$$ and note that $\tilde{x}$ satisfies all the conditions necessary to be in $D(A)$, including the $H^2(1,2)$ condition since on a symmetric interval around the only potentially problematic point $\xi=3/2$, the function $\tilde{x}$ is the negative reflection of an $H^2$ function around a point at which it is $0$. Morever, the classical solution to (\ref{WHWeq1}) of initial data $\tilde{x}$ is given by $$\tilde{x}(\xi,t) = \begin{cases} x_*(\xi,t), & \xi \in [0,3/2], \ t>0, \\ -x_*(3-\xi,t), & \xi \in (3/2,3], \ t>0,
	\end{cases}$$ where $x_*(\xi,t)$ is the classical solution to (\ref{WHWeq2}) for intial data $x_*$. It follows that
	$$E_{\tilde{x}}(t) = 2E_{x_*}(t) \neq o(r(t))$$ as $t \to \infty$.
	%and note that $\tilde{x}$ is clearly coupled at the points $1,2$ since $x_*$ is coupled at $1$.
	%Moreover, $\tilde{x}$ satisfies the regularity conditions to be in $D(A)$, even on the interval $[1,2]$ since 
\end{proof}

\section{Possible Future Directions}

In this final section, we pose and comment on a few questions about possible future directions arising out of systems similar to that described by (\ref{WHWeq1}). The last of these questions could potentially be very interesting and not easily tractable.

Note, however, that the question likely to be asked first -- what happens when the Dirichlet conditions are replaced by Neumann conditions -- is easily answered. In this case, the semigroup is actually unbounded. The function $x(\xi,t) = (at,a,at)$ for any constant $a\neq 0$ solves the variant of (\ref{WHWeq1}) where the fourth line is changed to $u_t(0,t)=\ut_t(3,t) = 0$ for the initial condition $(0,a,0)$, which yields an unbounded orbit of the semigroup in this case. That said, an alternative formulation involving a different state space can be chosen for the Neumann problem, one that is more physically intuitive and for which the same method as for the Dirichlet case can be applied to obtain the same rate of decay. With that out of the way, we ask the following natural two part question.
%The following is perhaps the most natural question to ask.
%
%\begin{opq}
%	What happens when the boundary conditions in (\ref{WHWeq1}) are changed from Dirichlet to Neumann?
%\end{opq}
%In this case, one would expect to be able to recreate a similar splitting to that in \cite[Proposition 2.4]{BPS1}. If this is possible, it seems likely that the optimal rate of energy decay would remain the same up to a change in constant (see also Section 6.6 below).

\begin{opq}
	\begin{enumerate}[(i)]
		\item Does the rate of energy decay remain the same up to a constant when extra wave and heat parts are added, for example, in a wave-heat-wave-heat or a wave-heat-wave-heat-wave system?
		\item If the energy remains optimally bounded by $Ct^{-4}$ for $C$ dependent on the particular system, can we find an explicit $N$-formula for the multiplicative constant $C_N$ bounding the energy of the system composed of $N$ wave-heat pairs all coupled together?
	\end{enumerate}
\end{opq}

Our first reaction at the thought of answering this question is one of horror, as the methods used in this article involved inverting a $4\times4$ matrix, and a system composed of $N$ wave-heat pairs would require the inversion of a $(4N-2)\times(4N-2)$ matrix. Though we believe that the answer to the first part of the above question is affirmative, the second part would require some clever matrix tricks to avoid total carnage. It is notable, however, that the matrices would have $0$ entries everywhere, except off of a diagonal of width at most four. So perhaps it is doable.

The idea of, perhaps inductively, obtaining a formula for $C_N$ as above leads to the question of homogenisation, that is, the computation of a limit equation. For $N \in \N$ and a given smooth function $f$, consider the following system of mixed hyperbolic and elliptic type that was studied in \cite{Wa1}:

\begin{equation}\label{HEHomo}
\begin{cases}\begin{aligned} & \partial^2_{t}u_N(\xi,t) - \partial^2_{\xi}u_N(\xi,t) = \partial_t f(\xi,t), & \xi \in \bigcup_{j \in \{1,\dots,N\}}\left(\frac{j-1}{N},\frac{2j-1}{2N}\right), \ & t\in \R,\\
& u_N(\xi,t) - \partial^2_{\xi}u_N(\xi,t) = f(\xi,t), & \xi \in \bigcup_{j \in \{1,\dots,N\}}\left(\frac{2j-1}{2N},\frac{j}{N}\right), \ & t\in \R,\\
& \partial_\xi u_N(0,t) = \partial_\xi u_N(1,t), & \ & t \in \R,
\end{aligned}\end{cases}\end{equation} subject to zero initial conditions and the requirement that the $u_N$ and their derivatives are continuous. We use the notation $\partial$ for derivatives as in \cite{Wa1} to avoid a mess involving the subscript $N$. Note that requiring conditions of continuity at the junction points results in the coupling considered in \cite{ZZ1} rather than that of \cite{BPS1} and (\ref{WHWeq1}).

Waurick showed in \cite{Wa1} that as $N\to \infty$, the sequence of solutions $(u_N)_{N\in \N}$ converges weakly in $L_{loc}^2([0,1]\times \R)$ to $u$, the solution to the limit equation
$$\frac{1}{2}\partial_t^2 u(\xi,t)+\partial_t u(\xi,t)+\frac{1}{2}u(\xi,t) - 2\partial_\xi^2 u(\xi,t) = f(\xi,t) + \partial_t f(\xi,t), \ \ \ t \in \R,$$ subject to zero initial conditions and Neumann boundary on both ends. Furthermore, he showed that this limit admitted exponentially stable solutions in the sense of \cite{Tros1}. However, when the elliptic part, $u_N(\xi,t) - \partial_\xi^2 u_N(\xi,t) = f(\xi,t)$, is replaced with the corresponding parabolic part, $\partial_t u_N(\xi,t) - \partial_\xi^2 u_N(\xi,t) = f(\xi,t)$, the limit equation becomes
$$\frac{1}{2}\partial_t^2 u(\xi,t)+\partial_t u(\xi,t)- 2\partial_\xi^2 u(\xi,t) = f(\xi,t) + \partial_t f(\xi,t), \ \ \ t \in \R,$$ subject again to zero initial conditions and Neumann boundary. Crucially, this limit equation is not exponentially stable, raising the following question.

\begin{opq}
	For the homogenised limit equation with mixed hyperbolic and parabolic parts as above, can the limit solution be posed and solved by a non-uniformly stable semigroup?
\end{opq}

In \cite{FrWa}, resolvent estimates of some kind are calculated in a way that depends on $N$ via the Gelfand transform, before a numerical analysis is conducted. How this might somehow be converted to a resolvent estimate for the limit problem itself remains an interesting unanswered question.

\section*{Appendix}

\subsection*{Entries for the Cofactor Matrix $C$ in Section \ref{resolventestimates}}

\addcontentsline{toc}{chapter}{Appendix}
\begin{align*} c_{11} & = -\lb^{3/2}[e^{\sqrt{\lambda}}T_+(\lb) + e^{-\sqrt{\lambda}}T_-(\lb)], \\
c_{12} & = \lambda^2\cosh(\lambda)(e^{\sqrt{\lambda}}T_+ - e^{-\sqrt{\lambda}}T_-), \\
c_{13} & = -\lambda^2\cosh(\lambda)[\esl T_+(\lb) +\nesl T_-(\lb)], \\
c_{14} & = \lb^{3/2} \cosh(\lambda), \\
c_{21} & = -\lambda[\esl T_+(\lb) - \nesl T_-(\lb)],\\
c_{22} & = -\lambda^2\sinh(\lambda)[\esl T_+(\lb) -\nesl T_-(\lb)],\\
c_{23} & = \lambda^2\sinh(\lambda)[\esl T_+(\lb) + \nesl T_-(\lb)],\\
c_{24} & =  -\lb^{3/2}\sinh(\lambda), \\
c_{31} & = \lb^{3/2}\cosh(\lambda), \\
c_{32} & = \lb^{5/2}\sinh(\lambda)\cosh(\lambda) \\
c_{33} & = \lambda^2\cosh^2(\lambda), \\
c_{34} & = -\lb^{3/2}[\esl T_+(\lb) \nesl T_-(\lb)],\\
c_{41} & = -\lb^{3/2}\sinh(\lambda), \\
c_{42} & = -\lb^{5/2} \sinh^2(\lambda), \\
c_{43} & = -\lambda^2\cosh(\lambda)\sinh(\lambda), \\
c_{44} & = -\lambda[\esl T_+(\lb) - \nesl T_-(\lb)].
\end{align*}

%\subsection*{Right Inverse Inequality in Proof of Proposition \ref{halfprop}}
%
%\begin{lemma*} If $f,g:(0,\infty)\to (0,\infty)$ are non-increasing functions and $f \leq g$, then $f_*^{-1} \leq g_*^{-1}$ on the interval $(0,f(0)]$ where $f_*^{-1}:(0,f(0)] \to [0,\infty)$ and $g_*^{-1}:(0,g(0)]\to [0,\infty)$ are right inverses of $f$ and $g$ respectively.
%\end{lemma*}
%
%\begin{proof} Either draw a picture or note that \begin{align*}
%	& f(t) \leq g(t) \text{ for all } t\in (0,\infty)\\
%	\implies & f(g_*^{-1}(s)) \leq s \text{ for all } s\in (0,f(0)]\\
%	\implies & f(g_*^{-1}(f(t))) \leq f(t) \text{ for all } t\in (0,\infty)\\
%	\implies & g_*^{-1}(f(t)) \geq t \ \text{ for all } t \in (0,\infty) \text{ since f is non-increasing}\\
%	\implies & g_*^{-1}(s) \leq f_*^{-1}(s) \text{ for all } s\in (0,f(0)].\\
%	\end{align*}\end{proof}
\providecommand{\bysame}{\leavevmode\hbox to3em{\hrulefill}\thinspace}
\providecommand{\MR}{\relax\ifhmode\unskip\space\fi MR }
\providecommand{\MRhref}[2]{%
	\href{http://www.ams.org/mathscinet-getitem?mr=#1}{#2}
}
\providecommand{\href}[2]{#2}

\end{document}